\documentclass{amsart}
\usepackage{graphicx}
\usepackage{amssymb}
\usepackage{amsmath}
\usepackage{amsthm}
\usepackage{subfigure}
\usepackage{enumerate}
\usepackage{bbm}
\usepackage{mathabx}
\usepackage{xcolor}

\newtheorem{thm}{Theorem}[section]
\newtheorem{cor}[thm]{Corollary}

\newtheorem{lem}[thm]{Lemma}
\newtheorem{prop}[thm]{Proposition}
\theoremstyle{definition}
\newtheorem{defn}[thm]{Definition}
\newtheorem{rem}[thm]{Remark}

\newtheorem*{defn*}{Definition}
\newtheorem*{rems*}{Remarks}
\newtheorem*{rem*}{Remark}

\numberwithin{equation}{section}

\newcommand{\SC}{\mathrm{SC}}
\newcommand{\Eq}{\mathrm{E}}
\newcommand{\change}[1]{\textcolor{black}{#1}}

\begin{document}

\title[The geometry of the secant caustic] {The geometry of the secant caustic \linebreak of a planar curve}
\author{W. Domitrz, M. C. Romero Fuster, M. Zwierzy\'nski}
\address{
  W. Domitrz, \textsc{Warsaw University of Technology, Faculty of Mathematics and Information Science, ul. Koszykowa 75, 00-662 Warszawa, Poland}\newline
  \textit{E-mail address}: \texttt{domitrz@mini.pw.edu.pl}
\newline\newline
\indent M. C. Romero Fuster, \textsc{University of Valencia, Faculty of Mathematics, 46100 Burjassot, Valencia, Spain},\newline
  \textit{E-mail address}: \texttt{carmen.romero@uv.es}
\newline\newline
 \indent M. Zwierzy\'nski, \textsc{Warsaw University of Technology, Faculty of Mathematics and Information Science, ul. Koszykowa 75, 00-662 Warszawa, Poland}\newline
  \textit{E-mail address}: \texttt{zwierzynskim@mini.pw.edu.pl}
}

\thanks{The work of M. C. Romero Fuster was partially supported by DGCYT and FEDER grant no. MTM2015-64013-P}



\begin{abstract}
The secant caustic of a planar curve $M$ is the image of the singular set of the secant map of $M$. We analyse the geometrical properties of the secant caustic of a planar curve, i.e. the number of branches of the secant caustic, the parity of the number of cusps and the number of inflexion points in each branch of this set. In particular, we investigate in detail some of the geometrical properties of the secant caustic of a rosette, i.e. a smooth regular oriented closed curve with non-vanishing curvature.
\end{abstract}

\maketitle

\section{Introduction}

Given a pair of closed smooth  curves $M, N$ respectively parameterized by $f, g: S^1 \rightarrow \mathbb R^n$, we can attach to it a smooth map $S_{M, N}: S^1 \times S^1 \rightarrow  \mathbb R^n$, known as the \textit {secant map} of  $(M, N)$. This is obtained by attaching to each couple $(s,t) \in S^1 \times S^1$   the end point of the vector $f(s)-g(t)$.  The study of the singular set  of the secant map provides relevant global information on the curves, from both the geometrical and topological viewpoint.  In fact, it is easy to see that $(s,t)$ is a singular point of the secant map if and only if the vectors $f'(s)$ and $g'(t)$ are parallel. The case of curves in $3$-space was analyzed by J. W. Bruce \cite{Bruce}, who proved that the secant map of a generic pair of space curves is a locally stable map from $\mathbb{R}^2$ to $\mathbb{R}^3$, having a cross-cap point at every pair of points with parallel tangents. It is not difficult to check that the singular set of the secant map coincides with the bitangency curves studied in \cite{Nuno-Romero}.

On the other hand, in the case of plane curves, it has been shown in \cite{RF-S} that a singular point is of a cusp type (or worse) if and only if  the curves $\alpha$ and $\beta$ have the same curvature and bend in opposite directions at the points $f(s)$ and $g(t)$. A description of the singularities of these maps, up to codimension $2$, together with their corresponding geometrical interpretation, is given in \cite{RF-S}. We quote the following results:

\begin{itemize}
\item[a)] The existence of a homotopically trivial connected component in the singular set of the secant map implies that both curves are non convex (in the sense that the curvature changes the sing along them).
\item[b)] For most  pairs  of closed curves $(M, N)$  (i.e. for an open and dense subset of $C^{\infty}(S^1,  \mathbb R^2) \times C^{\infty}(S^1,  \mathbb R^2)$ with  the Whitney $C^{\infty}$-topology), the  map  $S_{M, N}$ is a stable map from the torus to the plane. In such case, we said that  $(M, N)$ is a \textit{stable pair of curves}.
\item[c)] Given a stable pair of  $(M, N)$  with  respective Whitney  indices  $m$ and $n$,  the singular set of  $S_{M, N}$ has exactly $2 \mu _{n,m}$  toric curves of type $(\frac{n}{\mu _{n,m}},\frac{m}{\mu _{n,m}}) $, where $ \mu _{n,m}$ denotes the maximum common divisor of $m$ and $n$ and possibly some  more homotopically trivial toric curves.
\end{itemize}

The  image of the singular set (branch set) of a stable pair $(M, N)$ is a  (non necessarily connected) plane curve  with possible cusps and transverse self-intersections. We shall denominate it as the \textit{secant caustic} of the curve and denote it as  $\SC(M)$.
The aim of the present paper is to analyze the geometrical properties of the secant caustic of a generic closed plane curve, i.e. we shall consider  the particular case  $\alpha = \beta$. We observe that in this case the diagonal of $S^1\times S^1$ belongs to the singular set of the secant map (i.e. all the pairs $(s,s)$ are singular), but its image reduces to the origin of $\mathbb R^2$. This is a degenerate component of the secant map that does not provide any information and shall not be considered here.

We must  emphasize the connections between the secant caustic, the Wigner caustic and the Centre Symmetry Set. The Wigner caustic was introduced by M. Berry in 1977 \cite{Berry} (\change{the first time it was called the Wigner caustic in \cite{Ozorio-Hannay}}) and the Centre Symmetry Set was introduced by S. Janeczko in 1996 \cite{J1}. From a geometrical viewpoint,  the Wigner caustic of a plane curve $M$ is the locus of the midpoints of all the chords connecting couples of points  with parallel tangent lines on $M$. The Centre Symmetry Set is the envelope of these chords. The local geometrical properties of the Wigner caustic and the Centre Symmetry Set have been analyzed in \cite{Ozorio-Hannay, Craizer, CDR1, CDR2, DJRR1, DMR1, DR1, DRR1, D-Z2, D-Z3, GH1, GWZ1, GZ, J1, S1, S2}. The global properties of the Wigner caustic of  closed planar curves were studied in \cite{D-Z, D-Z3}.  For this purpose, in \cite{D-Z} it was introduced an algorithm (``glueing schemes'')  encoding the information relative to the connection between the  different smooth branches of the Wigner caustic was introduced. This algorithm allows us to determine the number of smooth branches, the rotation number, the number of inflexion points and the parity of the number of cusps on each one of these branches.  In the present paper (Section 3), we adapt this algorithm to the analysis of the secant maps of  closed plane curves and as a result we obtain several global consequences regarding the behaviour of the branch set of these maps. In particular we point out the following results:

\begin{itemize}
\item The numbers of cusps and of inflexion points of the secant caustic of a generic closed regular plane curve are even.

\item Let $ a, b$ be a parallel pair of a generic closed regular plane curve $M$. Then $a-b$ and $b-a$ are inflexion points of the secant caustic of $M$ if and only if one of the points $a$ or $b$ is an inflexion point of $M$.

\item  Given a regular curve $M$ and a parallel pair $a,b$ of $ M$, the secant caustic $\SC(M)$  of $M$ has a singularity at the point $a-b$ if and only if $M$ is curved in the same side at $a$ and $b$ and  $|\kappa_{ M}(a)|=|\kappa_{ M}(b)|$.

\item The secant caustic of a regular  closed curve $M$ passes through the origin at the inflexion points of $M$.
\end{itemize}

\change{The secant caustic of a regular curve is equivariant under affine transformations. Hence, the number of singular points and the number of inflexion points of the secant caustic are affine invariants. The techniques used in the paper are based on the Euclidean geometry but the Euclidean metric on the plane and  a parametrization of a curve are only technical tools that allow us to apply the methods of differential geometry and singularity theory. Main results do not depend on the parametrization neither the Euclidean metric.}


\section{The geometry of the secant caustic}

Let $M$ be a smooth parametrizable curve on the affine plane $\mathbb{R}^2$, i.e. the image of the $C^{\infty}$ smooth map from an interval to $\mathbb{R}^2$. A smooth curve is \textit{closed} if it is the image of a $C^{\infty}$ smooth map from $S^1$ to $\mathbb{R}^2$. A point of a curve is called \textit{regular} if the velocity of the curve does not vanish at this point -- otherwise it is called \textit{singular}. If all points of a curve are regular, the curve is called \textit{regular}. A regular curve is said to be \textit{simple} if it has no self-intersection points. A regular simple closed curve is \textit{convex} provided its curvature does not vanish, \change{i.e. if the order of contant of the curve with its tangent lines at any point is one. It is obvious that the order of contact of the curve with its tangent line is affine invariant}. \change{A \textit{rotation number} of a regular curve is a rotation number of its velocity vector field.}

\change{Throughout the article we will assume that the considered curves are regularly parametrizable.}

\change{Let $M$ be a regular curve or the union of two regular curves.}

\begin{defn}\label{parallelpair}
\change{A~pair of points $a,b\in M$, where $a\neq b$, is called a \textit{parallel pair} if the tangent line to $M$ at $a$ is parallel to the tangent line to $M$ at $b$.}
\end{defn}

\begin{defn}\label{chord}
The \textit{chord} connecting  a pair of points $a,b\in M$, is the line:
\begin{align*}
l(a,b)=\left\{\lambda a+(1-\lambda)b\ \big|\ \lambda\in\mathbb{R}\right\}.
\end{align*}
\end{defn}

\begin{defn}\label{equidistantSet}
\change{For a fixed value of $\lambda\in\mathbb{R}$, an \textit{affine $\lambda$-equidistant} of $M$ is the following set:}
\begin{align*}
\Eq_{\lambda}(M)=\textrm{cl}\left\{\lambda a+(1-\lambda)b\ \big|\ a,b \text{ is a parallel pair of } M\right\},
\end{align*}
where $\mathrm{cl}X$ is the closure of $X$. In particular, the set $\Eq_{\frac{1}{2}}(M)$ will be called the \textit{Wigner caustic} of $M$.
\end{defn}

\begin{defn}
The \textit{secant caustic} of $M$, is the following set:
\begin{align*}
\SC(M)=\textrm{cl}\left\{a-b\ \big|\ a,b \text{ is a parallel pair of } M\right\}.
\end{align*}
\end{defn}

See Fig. \ref{PictureSCexample1} for examples of secant caustics.

\begin{rem}
The closure in the previous definition is needed to include the origin $\bold{0}$ of $\mathbb{R}^2$ in the secant caustic of $M$ as a limit of the differences $a-b$ of parallel pairs of points $a,b$ approaching an inflexion point of $M$ from different sides.
\end{rem}

\begin{rem}
Let $\displaystyle\omega=\sum_{i=1}^{n}\mathrm{d}p_i\wedge \mathrm{d}q_i$ be the canonical symplectic form on $\mathbb{R}^{2n}$. The map $\flat: T\mathbb{R}^{2n}\ni v\mapsto \omega(v, \cdot)\in T^\ast\mathbb{R}^{2n}$ is an isomorphism between the bundles $T\mathbb{R}^{2n}$ and $T^\ast\mathbb{R}^{2n}$. Let $\alpha$ be the canonical Liouville $1$-form on $T^\ast\mathbb{R}^{2n}$. Then $\displaystyle\dot{\omega}=\flat^\ast\mathrm{d}\alpha=\sum_{i=1}^{n}\mathrm{d}\dot{p}_i\wedge \mathrm{d}q_i+\mathrm{d}p_i\wedge\mathrm{d}\dot{q}_i$ is a symplectic form on $T\mathbb{R}^{2n}$. 
It is easy to see that the map $\Psi:\mathbb{R}^{2n}\times\mathbb{R}^{2n}\to T\mathbb{R}^2=\mathbb{R}^{2n}\times\mathbb{R}^{2n},$
\begin{align*}
\Psi(p^+, q^+, p^-, q^-)=\left(\frac{p^++p^-}{2}, \frac{q^++q^-}{2}, p^+-p^-, q^+-q^-\right)
\end{align*}
takes the canonical symplectic form $\pi_+^\ast\omega-\pi_-^\ast\omega$ on the product $\mathbb{R}^{2n}\times\mathbb{R}^{2n}$ to the canonical symplectic form $\dot{\omega}$ on $T\mathbb{R}^{2n}$, where $\pi_+, \pi_-:\mathbb{R}^{2n}\times\mathbb{R}^{2n}\to\mathbb{R}^{2n}$ are the projections on the first and on the second component, respectively. Let $L$ be a Lagrangian submanifold of $(R^{2n}, \omega)$, then $\Psi(L\times L)$ is a Lagrangian submanifold of $(T\mathbb{R}^{2n}, \dot{\omega})$. Let $\pi_1, \pi_2:T\mathbb{R}^{2n}=\mathbb{R}^{2n}\times\mathbb{R}^{2n}\to\mathbb{R}^{2n}$ be the projections on the first and on the second component, respectively. Then $\pi_1$ and $\pi_2$ define Lagrangian fibre bundles with the symplectic structure $\dot{\omega}$. Let us notice that the caustic of the Lagrangian map $\pi_1\circ\Psi\big|_{L\times L}$ is the Wigner caustic \cite{CDR1, DMR1, DR1}. On the other hand the Lagrangian map  $\pi_2\circ\Psi\big|_{L\times L}$ is the secant map of $L$. Therefore the set of singular values of the secant map coincides with the secant caustic. \change{Finally, observe that any smooth regular planar curve is an immersed Lagrangian submanifold of $(R^2, \omega)$.}
\end{rem}

\begin{figure}[h]
\centering
\includegraphics[scale=0.4]{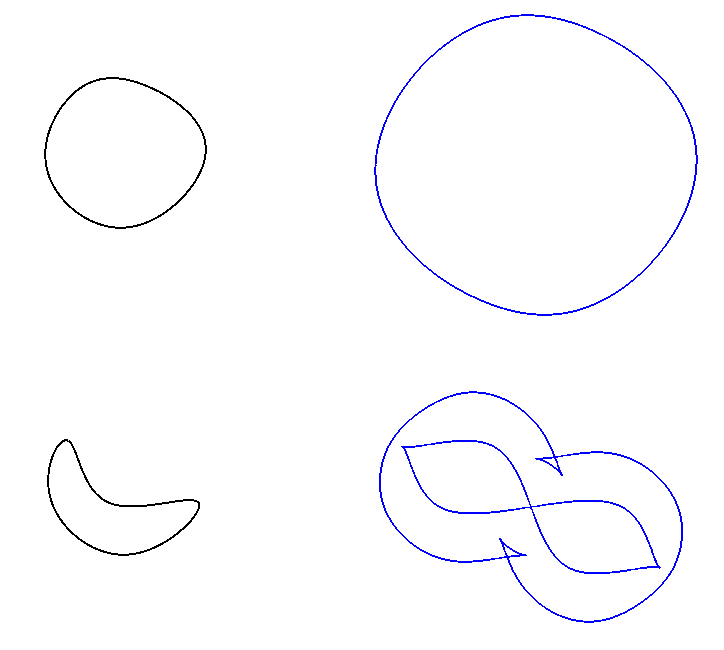}
\caption{Curves (on the left) and the secant caustic of them (on the right)}
\label{PictureSCexample1}
\end{figure}

\begin{defn}
\change{Let $f$ be a regular parameterization of a curve $M$. We say that $M$ is \textit{parameterized in the same direction} at points $f(s_1)$ and $f(s_2)$ if $f'(s_1)=\alpha f'(s_2)$, where $\alpha>0$.}
\end{defn}

\begin{defn}\label{DefNonDegInf}
A point $p$ of a smooth regular curve $M$ is called an \textit{inflexion point} of $M$ if the curvature of $M$ changes sign at $p$, and if the curvature vanishes at the point $p$ but does not change sign, then the point $p$ is called an \textit{undulation point} of $M$. An inflexion point $f(s_0)$ of $M$ parameterized by $f$ is \textit{non-degenerate} if $\det\big(f'(s_0),f'''(s_0)\big)\neq 0$.
\end{defn}

\begin{defn}\label{DefCusp}
Let $I,J\subset\mathbb{R}$ be open intervals such that $0\in I$, and $U, V\subset\mathbb{R}^2$ be open sets. A singular point $f(s_0)$, where $s_0\in J$, of a curve $f\in C^{\infty}(J, U)$, is called a \textit{cusp} if there exist diffeomorphisms $\varphi: I\to J$ and $\Phi: U\to V$ such that $\varphi(0)=s_0$ and
$$\left(\Phi\circ f\circ\varphi\right)(t)=\big(t^2,t^3\big).$$
\end{defn}

Let us denote by $\kappa_{ M}(p)$ the signed curvature of $ M$ at $p$. 

\change{Let $a,b$ be a parallel pair of $M$ and let $\kappa_M(b)\neq 0$. Let $f:(s_0,s_1)\to\mathbb{R}^2$, $g:(t_0,t_1)\to\mathbb{R}^2$ be local arc length parameterizations of $M$ nearby the points $a,b$ such that these parameterizations are in the same direction. Then there exists a function $t:(s_0,s_1)\to(t_0,t_1)$ such that
\begin{align}\label{ParallelPropertyLemma212}
f'(s)=g'(t(s)).
\end{align}
By the implicit function theorem the function $t$ is smooth and 
\begin{align}
\displaystyle t'(s)=\frac{\kappa_{M}(f(s))}{\kappa_{M}(g(t(s)))}.
\end{align}
Then a \textit{local natural parameterization of the secant caustic} is given in the following way
\begin{align}
\label{SCParamLemma}
\gamma_{\mathcal{SC}}(s)=f(s)-g(t(s)).
\end{align}
Whenever we will talk about singular points of the secant caustic we will understand these points as the singular points of the parameterization given by \eqref{SCParamLemma}.}

\begin{lem}\label{LemParallelCurvature}
Let $ M$ be a closed smooth regular curve. Let $a,b$ be a parallel pair of $ M$, such that $M$ is arc length parameterized at $a$ and $b$ in the same direction and $\kappa_{ M}(b)\neq 0$. Let $p=a-b$. Then 
\begin{enumerate}[(i)]
\item if $p$ is regular, then the tangent line of $\SC( M)$ at $p$ is parallel to the tangent lines of $ M$ at $a$ and $b$, $\kappa_{M}(a)\neq \kappa_{M}(b)$, and the curvature of $\SC( M)$ at $p$ is equal to
\begin{align}
\kappa_{\SC( M)}(p)=\frac{\kappa_{ M}(a)|\kappa_{ M}(b)|}{|\kappa_{ M}(a)-\kappa_{ M}(b)|}.
\end{align}
\item if $p$ is singular, then $\kappa_{M}(a)=\kappa_{M}(b)$ and the point $p$ is a cusp if and only if $\kappa_{M}'(a)\neq\kappa_M'(b)$, where $'$ denotes the derivative with respect to the arc length parameter.
\end{enumerate}
\end{lem}
\begin{proof}

\change{By direct calculations using the natural parameterization of the secant caustic \eqref{SCParamLemma} we get (i). Theorem 3 in \cite{RF-S} implies (ii).}
\end{proof}

\begin{rem}
\change{Under the assumptions of Lemma \ref{LemParallelCurvature}, by direct calculations the ratio of curvatures $\displaystyle\frac{\kappa_{M}(a)}{\kappa_{M}(b)}$, where $a,b$ is a parallel pair, is an affine invariant. Let $f$ be the arc length parameterization  of $M$ such that $f(s),f(t(s))$ are parallel pairs, where $t(s)$ is defined by (\ref{ParallelPropertyLemma212}). Then it is obvious that the following property
\begin{equation}\label{diff_ratio}
\dfrac{\mathrm{d}}{\mathrm{d}s}\left(\dfrac{\kappa_M(f(s))}{\kappa_M(f(t(s)))}\right)\ne 0
\end{equation}
does not depend on the parameterization of $M$. But on the other hand, 
\begin{align}
\label{EqDiffRatioCurv}
\dfrac{\mathrm{d}}{\mathrm{d}s}\left(\dfrac{\kappa_M(f(s))}{\kappa_M(f(t(s)))}\right)=\dfrac{\kappa_M'(f(s))\kappa_M^2(f(t(s)))-\kappa_M'(f(t(s)))\kappa_M^2(t(s))}{\kappa_M^3(f(t(s)))},
\end{align}
where $'$ denotes the derivative with respect to the arc length parameterization. Thus, by \eqref{EqDiffRatioCurv}, a parallel pair $f(s_0), f(t(s_0))$, such that $\kappa_{M}(f(s_0))=\kappa_{M}(f(t(s_0))$ and 
\begin{equation}\label{diff_ratio}
\kappa_{M}'(f(s_0))\neq \kappa_{M}'(f(t(s_0)))
\end{equation}
 is  affine equivariant.}
\end{rem}

Let $C^{\infty}(S^1,\mathbb{R}^2)$ denote the topological space  of $C^{\infty}$  smooth maps from $S^1$ to $\mathbb R^2$ with the Whitney $C^{\infty}$ topology.
Similarly like in \cite{D-Z} we will find a generic subset $G$ of  $C^{\infty}(S^1,\mathbb{R}^2)$ such that for any curve $f$ in $G$ the secant caustic of $f$ is the finite union of smooth curves with at most cusp singularities. Wherever in the paper we say the curve $f$ is generic we mean that $f$ belongs to $G$ i. e.  it satisfies assumptions (i)-(v) of Theorem \ref{ThmGenericCS}. 

\begin{thm}\label{ThmGenericCS}
\change{Let $G$ be the subset of $C^{\infty}(S^1,\mathbb{R}^2)$ such that any $f\in G$ has the following properties.
\begin{enumerate}[(i)]
\item The curve $f$ is a regular curve.
\item The curve $f$ has only non-degenerate inflexion points and it has no undulation points.
\item The curve $f$ has only transversal self crossings.
\item If $f(s_1), f(s_2)$ is a parallel pair of $f$, then at least one of $f(s_1)$, $f(s_2)$ is not an inflexion point.
\item If $f$ is parameterized in the same direction at the points $f(s_1)$ and $f(s_2)$, and $\kappa_f(s_1)=\kappa_f(s_2)$, then $\kappa'_f(s_1)\neq\kappa'_f(s_2)$, where $\kappa_f'$ denotes the derivative of the curvature with respect to the arc length parameter.
\end{enumerate}
Then the set $G$ is a generic subset of the topological space $C^{\infty}(S^1,\mathbb{R}^2)$. Furthermore, the secant caustic of $f\in G$ is the finite union of smooth curves with at most cusp singularities.}
\end{thm}
\begin{proof}
\change{It is sufficient to justify that each of the properties is generic. We will use the Thom Transversality Theorems -- the classic one and one for multijets (see Theorems 4.9 and 4.13 in \cite{GolG} for details). The set $J^k_s(S^1,\mathbb{R}^2)$ is the \textit{$s$-fold $k$-jet bundle} and let $(S^1)^{(2)}$ be the set $\big(S^1\times S^1\big)\setminus\big\{(s,s)\,|\, s\in S^1\big\}$.}

\change{Note that the set of $1$-jets of smooth singular (i.e. not regular) closed curves is a smooth sub\-ma\-nifold of $J^1(S^1,\mathbb{R}^2)$ of codimension $2$. Hence, the subset of $C^{\infty}(S^1,\mathbb{R}^2)$ of regular curves is generic in  $C^{\infty}(S^1, \mathbb{R}^2)$.}

\change{Let $f\in C^{\infty}(S^1,\mathbb{R}^2)$ be regular. Property (ii) follows from the transversality of the map $j^3f:S^1\to J^3(S^1,\mathbb{R}^2)$ to the submanifold of $J^{3}(S^1, \mathbb{R}^2)$ defined as
\begin{align*}
\big\{j^3h(s)\in J^3(S^1, \mathbb{R}^2)\, \big|\, h'(s)\neq 0, \det\big(h'(s),h''(s)\big)=0\big\}.
\end{align*}
It is easy to calculate that this tranvsersality means that if $\det\big(f'(s),f''(s)\big)=0$, then $\det\big(f'(s), f'''(s)\big)\neq 0$, which is equivalent to the fact that an inflexion point is non-degenerate and the curve $f$ has no undulation points (see Definition \ref{DefNonDegInf}). Therefore, Property (ii) is generic by the Thom Transversality Theorem.}

\change{One can show that the transversality of $j_2^1f: (S^1)^{(2)}\to J_2^1(S^1,\mathbb{R}^2)$ to the submanifold
\begin{align*}
\Big\{\big(j^1g(s_1), j^1h(s_2)\big)\in J_2^1(S^1,\mathbb{R}^2)\,\Big|\, g'(s_1)\neq 0, h'(s_2)\neq 0, g(s_1)=h(s_2)\Big\}
\end{align*}
is equivalent to the property that if $f(s_1)=f(s_2)$ for $s_1\neq s_2$, then \linebreak $\det\big(f'(s_1),f'(s_2)\big)\neq 0$. Hence, Property (iii) is generic by the Thom Transversality Theorem for multijets.}

\change{Next, let's notice that the transversality of $j_2^2f: (S^1)^{(2)}\to J_2^2(S^1, \mathbb{R}^2)$ to the sumbanifold
\begin{align*}
\Big\{\big(j^2g(s_1), j^2h(s_2)\big)\in J_2^2(S^1,\mathbb{R}^2)\,\Big|\, g'(s_1)\neq 0, h'(s_2)\neq 0, \det\big(g'(s_1),h'(s_2)\big)=0\Big\}
\end{align*}
means that if $\det\big(f'(s_1),f'(s_2)\big)=0$ for $s_1\neq s_2$, then $\kappa_f^2(s_1)+\kappa_f^2(s_2)>0$. Therefore, Property (iv) is generic.}

\change{Now, let's assume that $f$ satisfies Property (iv). By direct calculations one can show that transversality of $j_2^3f: (S^1)^{(2)}\to J_2^3(S^1, \mathbb{R}^2)$ to the submanifold
\begin{align*}
\Big\{\big(j^3g(s_1), j^3h(s_2)\big)\in & J_2^3(S^1,\mathbb{R}^2)\,\Big|\, \det\big(g'(s_1), h'(s_2)\big)=0, g'(s_1)\neq 0, h'(s_2)\neq 0, \\ 
& g'(s_1)\cdot h'(s_2)>0, \kappa_g(s_1)=\kappa_h(s_2)\Big\}
\end{align*}
means exactly that the function $f$ satisfies Property (v), which ends the proof.}
\end{proof}

Note that any generic curve satisfies the assumptions of Proposition \ref{PropFiniteRotNumber}.

\begin{prop}\label{PropFiniteRotNumber}
\change{Let $M$ be a smooth regular closed curve. If $M$ has only non-degenerate inflexion points and no undulation points, then the rotation number of $M$ and the number of inflexion points of $f$ are finite.}
\end{prop}
\begin{proof}
\change{Let us assume that $f$ has only non-degenerate inflexion points and no undulation points. We will show that the number of inflexion points of $f$ is finite. Contrary to what we want to prove, suppose that the number of these points is not finite. Therefore, we can choose a sequence $\{s_n\}$ of non-degenerate inflexion points such that $s_n\to s_0$. Since
\begin{align*}
\det\big(f'(s_0),f''(s_0)\big)=\lim_{n\to\infty}\det\big(f'(s_n),f''(s_n)\big)=0,
\end{align*} 
the point $f(s_0)$ is also a non-degenerate inflexion point. Hence, 
\begin{align*}
\det\big(f'(s_0),f'''(s_0)\big) = 
\lim_{n\to\infty}\frac{1}{s_n-s_0}\det\big(f'(s_0),f''(s_n)-f''(s_0)\big)= \\ 
=\lim_{n\to\infty}\det\left(\frac{f'(s_0)-f'(s_n)}{s_n-s_0},f''(s_n)\right)=
-\det\big(f''(s_0),f''(s_0)\big)=0,
\end{align*}
which is a contradiction.}

\change{Let the rotation number of $f$ be infinite and let $f$ be the arclength parameterization. Then, for every angle $\psi\in[0,2\pi)$, we can find a sequence $\{s_n^\psi\}$ such that $s_n^{\psi}\to s_0^\psi$, $f'(s_n^\psi)=(\cos\psi,\sin\psi)$, and $f''(s_n^\psi)=\kappa(s_n^\psi)(-\sin\psi,\cos\psi)$. Since $f'$, $f''$, $\kappa$ are continuous, we get that $f'(s_0^\psi)=(\cos\psi,\sin\psi)$ and $f''(s_0^\psi)=\kappa(s_0^\psi)(-\sin\psi,\cos\psi)$. Next, let's notice that
\begin{align*}
\kappa(s_0^\psi)=\det\big(f'(s_0^\psi),f''(s_0^\psi)\big)=\lim_{n\to\infty}\frac{1}{s_n^\psi-s_0^\psi}\det\big(f'(s_0^\psi), f'(s_n^\psi)-f'(s_0^\psi)\big)=0.
\end{align*}
Therefore, for each $\psi\in[0,2\pi)$ there exists $s_0^\psi$ such that $f(s_0^\psi)$ is an inflexion point. Since $f$ has only finite number of inflexion points, we get that $s_0^{\psi_1}=s_0^{\psi_2}$ for some $\psi_1\neq\psi_2$. But then
$$f'(s_0^{\psi_1})=\left(\cos\psi_1,\sin\psi_1\right)=\left(\cos\psi_2,\sin\psi_2\right)=f'(s_0^{\psi_2}),$$
which is impossible. Therefore, the rotation number of $f$ is finite.}
\end{proof}

By genericity of a curve and by Lemma \ref{LemParallelCurvature}(i) we get the following proposition.

\begin{prop}\label{CorInflofCS}
Let $a, b$ be a parallel pair of a generic regular closed curve $M$. Then $a-b$ and $b-a$ are inflexion points of $\SC( M)$ if and only if one of $a$, $b$ is an inflexion point of $M$.
\end{prop}

\begin{defn}
\change{We say that a point $c$ is a \textit{center of symmetry} of a curve $M$ if for any point $a$ in $M$ the point $2c-a$ belongs to $M$.}
\end{defn}

\begin{prop}\label{PropCenSym}
Given a closed regular curve $M$, the origin of $\mathbb{R}^2$ is the center of symmetry of the secant caustic $\SC(M)$.
\end{prop}
\begin{proof}
Let $a,b$ be a parallel pair of $M$. Then $a-b$ and $b-a$ belong to $\SC(M)$ and it is easy to see that $\bold{0}$ is the center of symmetry of $\SC(M)$.
\end{proof}

Proposition \ref{PropCenSym} implies the following corollary.

\begin{cor}\label{CorEvenNumberCuspsInfl}
If $M$ is a generic closed regular curve, then the number of cusps and the number of inflexion points of $\SC(M)$ are even.
\end{cor}

Later we will show that Corollary \ref{CorEvenNumberCuspsInfl} still holds for each branch of the secant caustic of a planar curve -- see Theorems \ref{ThmInflofSC} and \ref{ThmAlgEvenNoOfCusp}.

\begin{defn}\label{DefCurved}
\change{Let $a,b$ be a parallel pair of $ M$ and assume that $a$ and $b$ are not inflexion points. Let $\tau_{p}$ denote the translation by a vector $p\in\mathbb{R}^2$. Then we say that $M$ is \textit{curved in the same side at $a$ and $b$} (respectively \textit{curved in the different sides at $a$ and $b$}) if the germs of the curves $M$ and $\tau_{a-b}(M)$ at $a=\tau_{a-b}(b)$ are on the same side (respectively on the different sides) of the tangent line to $ M$ at $a$.}
\end{defn}

We illustrate Definition \ref{DefCurved} in Fig. \ref{FigCurved}.

\begin{figure}[h]
\centering
\includegraphics[scale=0.35]{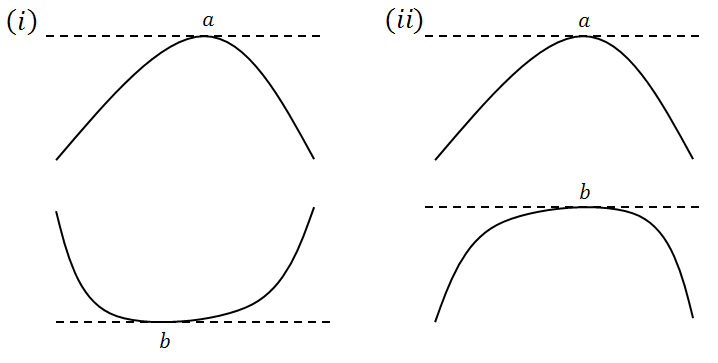}
\caption{(i) A curve curved in the different sides at a parallel pair $a, b$, (ii) a curve curved in the same side at a parallel pair $a, b$}
\label{FigCurved}
\end{figure}

By Lemma \ref{LemParallelCurvature} and Definition \ref{DefCurved} we obtain the following.

\begin{prop}\label{PropSingularPointOfBs}
Let $a,b$ be a parallel pair of $ M$. The curve $\SC(M)$ is singular at the point $a-b$ if and only if $M$ is curved in the same side at $a, b$ and  $|\kappa_{ M}(a)|=|\kappa_{ M}(b)|$.
\end{prop}

Let us notice that in the case of the Wigner caustic we have a similar result to Proposition \ref{PropSingularPointOfBs}, but we have to replace the phrase ''curved in the same side'' by ''curved in the different sides'' \cite{D-Z}.

\change{The Centre Symmetry Set of a curve $M$ has an asymptote if and only if $M$ is curved in the same side at a parallel pair $a,b$, and $\big|\kappa_{M}(a)\big|=\big|\kappa_{M}(b)\big|$ \cite{GH1}. Therefore, we have the following remark.}

\begin{rem}
A parallel pair $a, b$ of a curve $M$ gives rise to a singular point of the secant caustic if and only if the Centre Symmetry Set has an asymptote.
\end{rem}

\begin{prop}\label{PropPassOrigin}
Let $ M$ be a regular closed curve. At an inflexion point of $M$ the set $\SC(M)$ passes through the origin.
\end{prop}
\begin{proof}
It is a consequence of the definition of the secant caustic and the fact that nearby an inflexion point $p$ of a curve $M$ there exist sequences of parallel pair approaching $p$ from the different sides.
\end{proof}

\begin{rem}
\change{Generically the reciprocal of Proposition \ref{PropPassOrigin} is true by Theorem \ref{ThmGenericCS}(ii).}
\end{rem}

\begin{prop}
Let $ M$ be a generic oval. Then for a generic $\lambda\neq\frac{1}{2}$ we have:
\begin{align*}
\Eq_{\lambda}(\SC( M))=\SC(\Eq_{\lambda}( M)).
\end{align*}
\end{prop}
\begin{proof}
Since for each parallel pair $a, b$ of $M$, the curve $M$ is curved in the different sides at $a, b$, we get from Proposition \ref{PropSingularPointOfBs} that the set $\SC( M)$ is an oval. Let us notice that for a generic $\lambda\neq\frac{1}{2}$ the set $\Eq_{\lambda}(M)$ is a piecewise regular curve with at most cusp singularities \cite{GZ}. Hence we have well defined tangent line at any point of $\Eq_{\lambda}(M)$ and at any point of $\SC(M)$ (see also Definition \ref{DefTangentSC}).

One can check that both $\Eq_{\lambda}(\SC( M))$ and $\SC(\Eq_{\lambda}( M))$ are equal to the following set:
 \begin{align*}
\big\{(2\lambda -1)(a-b)\ \big|\ a,b\text{ is a parallel pair of }M\big\}.
\end{align*}
\end{proof}

\begin{thm}\label{ThmCurvPosWrinkled} 
Let $\mathcal{P}$, $\mathcal{Q}$ be embedded curves with end points $p_0, p_1$ and $q_0, q_1$ respectively and suppose that:
\begin{enumerate}[(i)]
\item The points $p_i, q_i$ form a parallel pair for $i=0,1$,
\item For every $q \in \mathcal{Q}$, there exists  $p\in \mathcal{P}$ such that $p, q$ is a parallel pair and if $p_i, q$ is a parallel pair then $q=q_i$ for $i=0,1$,
\item $\kappa_{\mathcal{P}}(p)>0$ for $p\neq p_0$, $\kappa_{\mathcal{Q}}(q_0)> 0$ and $\kappa_{\mathcal{Q}}(q_1)\geqslant 0$,
\item $\mathcal{P}$, $\mathcal{Q}$ are curved in the same side at parallel pairs $p,q$ close to $p_0,q_0$ and $p_1, q_1$, respectively.
\end{enumerate}

Then, provided the curvatures of $\mathcal{P}$ and $\mathcal{Q}$ satisfy the following condition
\begin{align}\label{CondCurvSings}
\Big(\kappa_{\mathcal{Q}}(q_0)-\kappa_{\mathcal{P}}(p_0)\Big)\cdot
\Big(\kappa_{\mathcal{Q}}(q_1)-\kappa_{\mathcal{P}}(p_1)\Big)<0,
\end{align}
the secant caustic of $\mathcal{P}\cup\mathcal{Q}$ has at least two singular points.
\end{thm}

\begin{proof} We shall use the method of the proof of Proposition 3.3 in \cite{D-Z3}
Let us assume that $\kappa_{\mathcal{P}}(p_0)>0$. 

Let $g:[t_0,t_1]\to\mathbb{R}^2, f:[s_0,s_1]\to\mathbb{R}^2$ be the arc length parameterizations of $\mathcal{P}, \mathcal{Q}$, respectively. From \text{(ii)-(iii)} we have that there exists a function $t:[s_0,s_1]\to[t_0,t_1]$ such that
\begin{align}\label{ParallelPropertyThm31}
f'(s)=g'(t(s)).
\end{align}
By the implicit function theorem the function $t$ is smooth and $\displaystyle t'(s)=\frac{\kappa_{\mathcal{Q}}(f(s))}{\kappa_{\mathcal{P}}(g(t(s)))}$. 
From (\ref{CondCurvSings}) we obtain that $\big(t'(s_0)-1\big)\cdot\big(t'(s_1)-1\big)<0$. Therefore by Darboux Theorem, there exists $s\in(s_0,s_1)$ such that $t'(s)=1$. Finally, from Proposition \ref{PropSingularPointOfBs} we get that the points $f(s)-g(t(s))$ and $g(t(s))-f(s)$ are singular points of $\SC\left(\mathcal{P}\cup\mathcal{Q}\right)$, which ends the proof in the case that $\kappa_{\mathcal{P}}(p_0)>0$

If we assume that $\kappa_{\mathcal{P}}(p_0)=0$, then the proof is similar except that the domain of the function $t'$ is $(s_0,s_1]$. Thus we replace $t'(s_0)$ by the limit of $t'$ at $s_0$.
\end{proof}

In Fig. \ref{FigWrinkled1} we illustrate two curves satisfying the assumptions of Theorem \ref{ThmCurvPosWrinkled}.

\begin{figure}[h]
\centering
\includegraphics[scale=0.3]{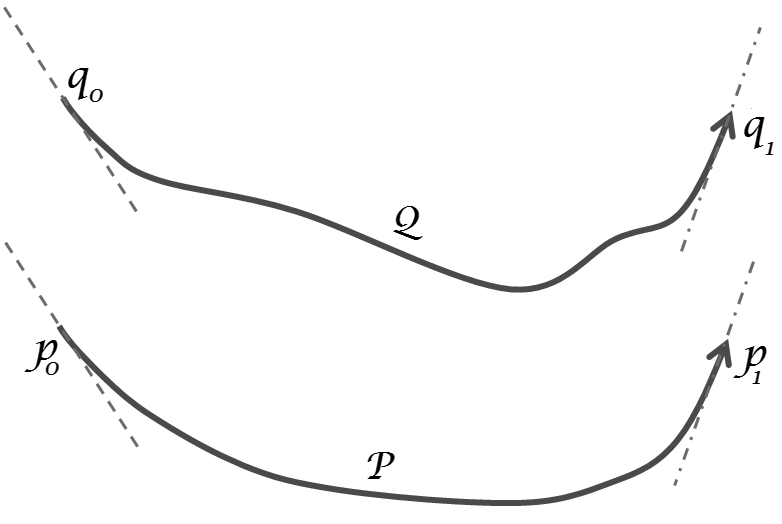}
\caption{Arcs satisfying assumptions of Theorem \ref{ThmCurvPosWrinkled}}
\label{FigWrinkled1}
\end{figure}

\begin{cor}\label{CorThmCurvPosWrinkled}
Under the assumptions (i)--(iv) of Theorem \ref{ThmCurvPosWrinkled}, provided $\kappa_{\mathcal{P}}(p_0)=\kappa_{\mathcal{Q}}(q_1)=0$ the secant caustic of $\mathcal{P}\cup\mathcal{Q}$ has at least two singular points.
\end{cor}

The result of Corollary \ref{CorThmCurvPosWrinkled} is illustrated in Fig. \ref{FigTwoInflPts}.

\begin{figure}[h]
\centering
\includegraphics[scale=0.25]{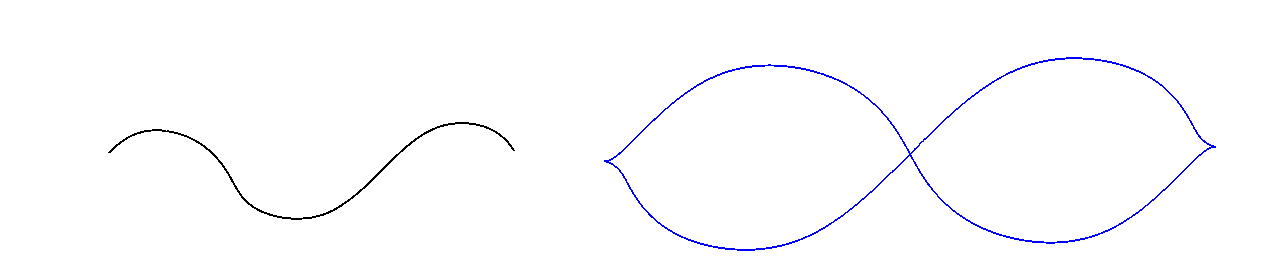}
\caption{A curve $M$ (on the left) and $\SC( M)$ (on the right)}
\label{FigTwoInflPts}
\end{figure}

In what follows, we shall denote  the translation by a vector $v$ as $\tau_{v}$.

\begin{thm}\label{pqTheorem2}
Let $\mathcal{P}$ and $\mathcal{Q}$ be embedded regular curves with endpoints $p_0$, $p_1$ and $q_0$, $q_1$, respectively. Let $l_0$ be a line through $q_0$ parallel to $T_{q_1}\mathcal{Q}$, $l_q$ be a line through $q_1$ parallel to $T_{q_0}\mathcal{Q}$ and $l_p'$ be a line through $\tau_{q_0-p_0}(p_1)$ parallel to $T_{q_0}\mathcal{Q}$. Denote  $c=l_p'\cap T_{q_1}\mathcal{Q}$, $b_0=l_0\cap l_p'$ and $b_1=T_{q_0}\mathcal{Q}\cap T_{q_1}\mathcal{Q}$ and let us assume that
\begin{enumerate}[(i)]
\item $T_{p_i}\mathcal{P}\| T_{q_i}\mathcal{Q}$ for $i=0,1$.
\item The curvatures of $\mathcal{P}$ and $\mathcal{Q}$ are positive.
\item The absolute values of rotation number of $\mathcal{P}$ and $\mathcal{Q}$ are the same and smaller than $\displaystyle\frac{1}{2}$.
\item For every point $p \in \mathcal{P}$, there exists a point $q \in \mathcal{Q}$ and for every point $q \in \mathcal{Q}$ there exists a point $p \in \mathcal{P}$, such that $p$, $q$ is a parallel pair.
\item $\mathcal{P}$ and $\mathcal{Q}$ are curved in the same side at every parallel pair $p, q$ such that $p\in\mathcal{P}$ and $q\in\mathcal{Q}$.
\end{enumerate}

Let $\rho_{\max}$ (respectively $\rho_{\min}$) be the maximum (respectively minimum) of the set
\begin{align*}
\left\{\frac{c-b_1}{q_1-b_1}, \frac{c-b_0}{\tau_{q_0-p_0}(p_1)-b_0}\right\}.
\end{align*}

Then provided $\rho_{\max}<1$ or $\rho_{\min}>1$, the secant caustic of $\mathcal{P}\cup\mathcal{Q}$ has at least two singular points.
\end{thm}

\begin{proof}
We shall use the method of the proof of Proposition 3.7 in \cite{D-Z3}. Let us consider the case $\rho_{\max}<1$, the proof of the case $\rho_{\min}>1$ is similar.

Let $A:\mathbb{R}^2\to\mathbb{R}^2$ be an affine transformation taking the parallelogram bounded by $T_{q_0}\mathcal{Q}$, $l_q$, $l_0$ $\tau_{q_0-p_0}\left(T_{p_1}\mathcal{P}\right)$ to the unit square. Let $\mathcal{P}'=A\left(\tau_{q_0-p_0}\left(\mathcal{P}\right)\right)$ and $\mathcal{Q}'=A\left(\mathcal{Q}\right)$ (see Fig. \ref{PictureThmParallelogram}). It is enough to show that $\SC(\mathcal{P}'\cup\mathcal{Q}')$ has at least two singular points. We shall take the coordinate system described in Fig. \ref{PictureThmParallelogram}(ii).

Let $L_{\mathcal{P}'}$, $L_{\mathcal{Q}'}$ be the lengths of $\mathcal{P}'$ and $\mathcal{Q}'$ respectively and take arc length parameterizations   of $\mathcal{P}'$ and $\mathcal{Q}'$, respectively given by $[0,L_{\mathcal{P}'}]\ni s\mapsto f(s)=\left(f_1(s),f_2(s)\right)$ and $[0,L_{\mathcal{Q}'}]\ni t\to g(t)=\left(g_1(t),g_2(t)\right)$ , such that $f(0)=g(0)=(0,0)$ and $g(L_{\mathcal{Q}'})=(g_1(L_{\mathcal{Q}'}),1)$, where $0<g_1(L_{\mathcal{Q}'})\leqslant \rho_{\max}$ and $f(L_{\mathcal{P}'})=(1,f_2(L_{\mathcal{P}'}))$, where $0<f_2(L_{\mathcal{P}'})\leqslant \rho_{\max}$.

\begin{figure}[h]
\centering
\includegraphics[scale=0.3]{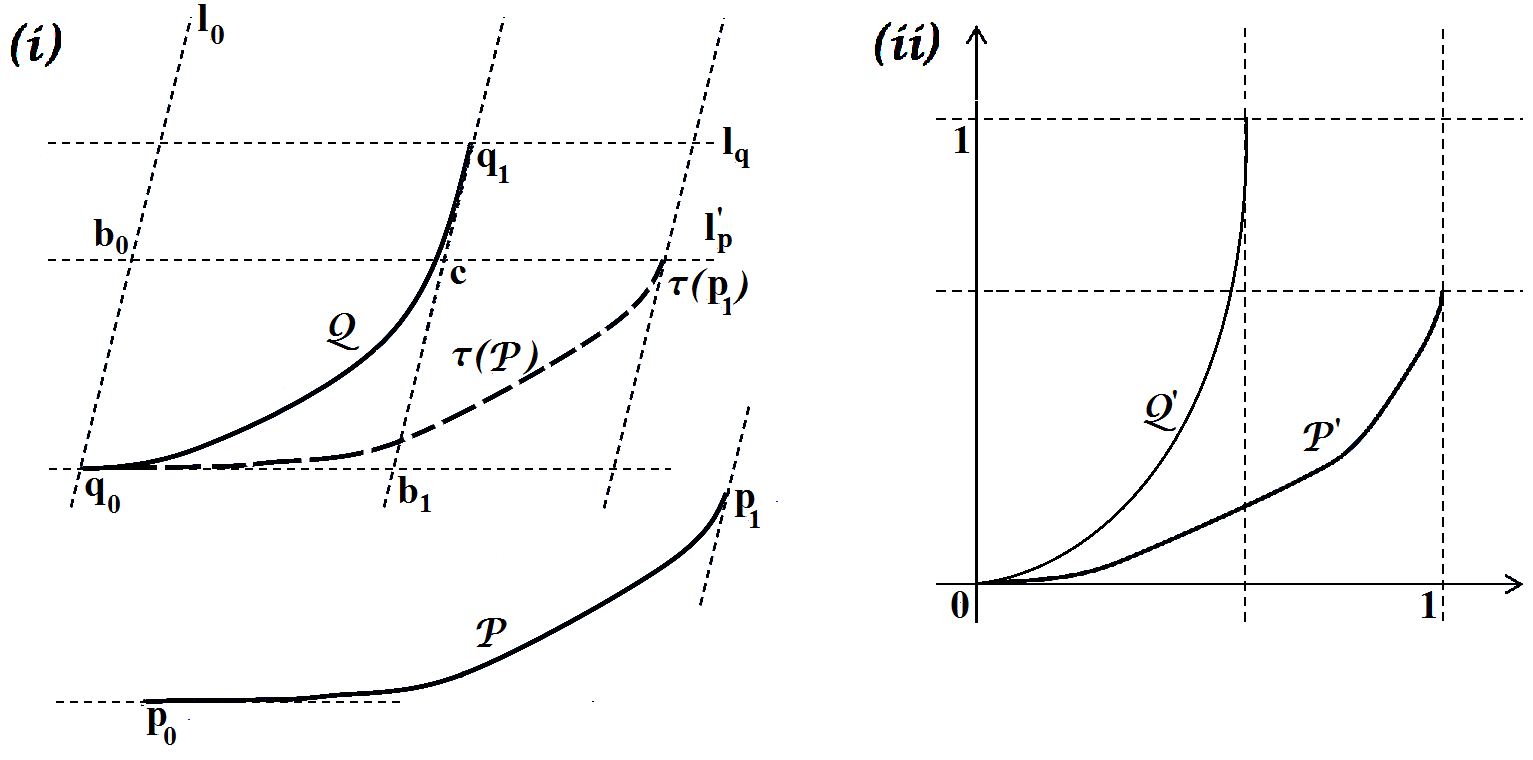}
\caption{Arcs satisfying assumptions of Theorem \ref{pqTheorem2}}
\label{PictureThmParallelogram}
\end{figure}

By the implicit function theorem we get that there exists a function $t:[0,L_{\mathcal{P}'}]\to [0,L_{\mathcal{Q}'}]$ such that 
\begin{align}
\label{EqParallel}
\frac{\mathrm{d}f}{\mathrm{d}s}(s)=\frac{\mathrm{d}g}{\mathrm{d}t}(t(s)).
\end{align}
This implies that $f(s),g(t(s))$ is a parallel pair. By (\ref{EqParallel}) we get that $\displaystyle t'(s)=\frac{\kappa_{\mathcal{P}'}(f(s))}{\kappa_{\mathcal{Q}'}(g(t(s)))}$.
Now from Proposition \ref{PropSingularPointOfBs} we have that the set $\SC(\mathcal{P}'\cup\mathcal{Q}')$ is singular if $\displaystyle\frac{\kappa_{\mathcal{P}'}(f(s))}{\kappa_{\mathcal{Q}'}(g(t(s)))}=1$ for some $s\in[0, L_{\mathcal{P}'}]$. Thus we need to show that $t'(s)=1$ for some $s\in[0,L_{\mathcal{P}'}]$. Let us assume that $t'(s)\neq 1$ for all $s\in[0,L_{\mathcal{P}'}]$. By (\ref{EqParallel}) we get that
\begin{align}\label{EquationGeGo}
g(t_e)-g(0)=\int_0^{L_{\mathcal{Q}'}}\frac{\mathrm{d}g}{\mathrm{d}t}(t)\,\mathrm{d}t=\int_0^{L_{\mathcal{P}'}}t'(s)\frac{\mathrm{d}g}{\mathrm{d}t}(t(s))\,\mathrm{d}s=\int_0^{L_{\mathcal{P}'}}t'(s)\frac{\mathrm{d}f}{\mathrm{d}s}(s)\,\mathrm{d}s.
\end{align}

Let us assume that $\displaystyle t'(s)>1$ for all $s\in[0, L_{\mathcal{P}'}]$. At the first component of (\ref{EquationGeGo}) we have 
\begin{align*}
g_1(L_{\mathcal{Q}'})=\int_0^{L_{\mathcal{P}'}}t'(s)\frac{\mathrm{d}f_1}{\mathrm{d}s}(s)\,\mathrm{d}s>\int_0^{L_{\mathcal{P}'}}1\cdot \frac{\mathrm{d}f_1}{\mathrm{d}s}(s)\,\mathrm{d}s=1.
\end{align*}
Then $g_1(L_{\mathcal{Q}'})>1$ which is impossible.

If we assume that $\displaystyle t'(s)<1$ for all $s\in[0,L_{\mathcal{P}'}]$, we obtain in a similar way that $1<f_2(L_{\mathcal{P}'})$, which is also impossible.

Therefore both $f(s)-g(t(s))$ and $g(t(s))-f(s)$ are singular points of  $\SC(\mathcal{P}'\cup\mathcal{Q}')$.
\end{proof}


\newcommand{\M}{M}
\newcommand{\C}{C}
\newcommand{\minm}{\mathbbm{m}}
\newcommand{\maxm}{\mathbb{M}}
\newcommand{\overarc}[2]{\underline{#1\frown #2}}
\newcommand{\p}{p}


\section{An algorithm to describe the geometry of the branches of the secant caustic}

Let $M$ be a generic regular closed curve. The secant caustic of $M$ is a union of smooth parametrized curves. Each of these curves we will be called a \textit{smooth branch} of $\SC( M)$. We illustrate a non-convex curve $M$ and different smooth branches of $\SC( M)$ in Fig. \ref{PictureAlgExam}.

\begin{figure}[h]
\centering
\includegraphics[scale=0.35]{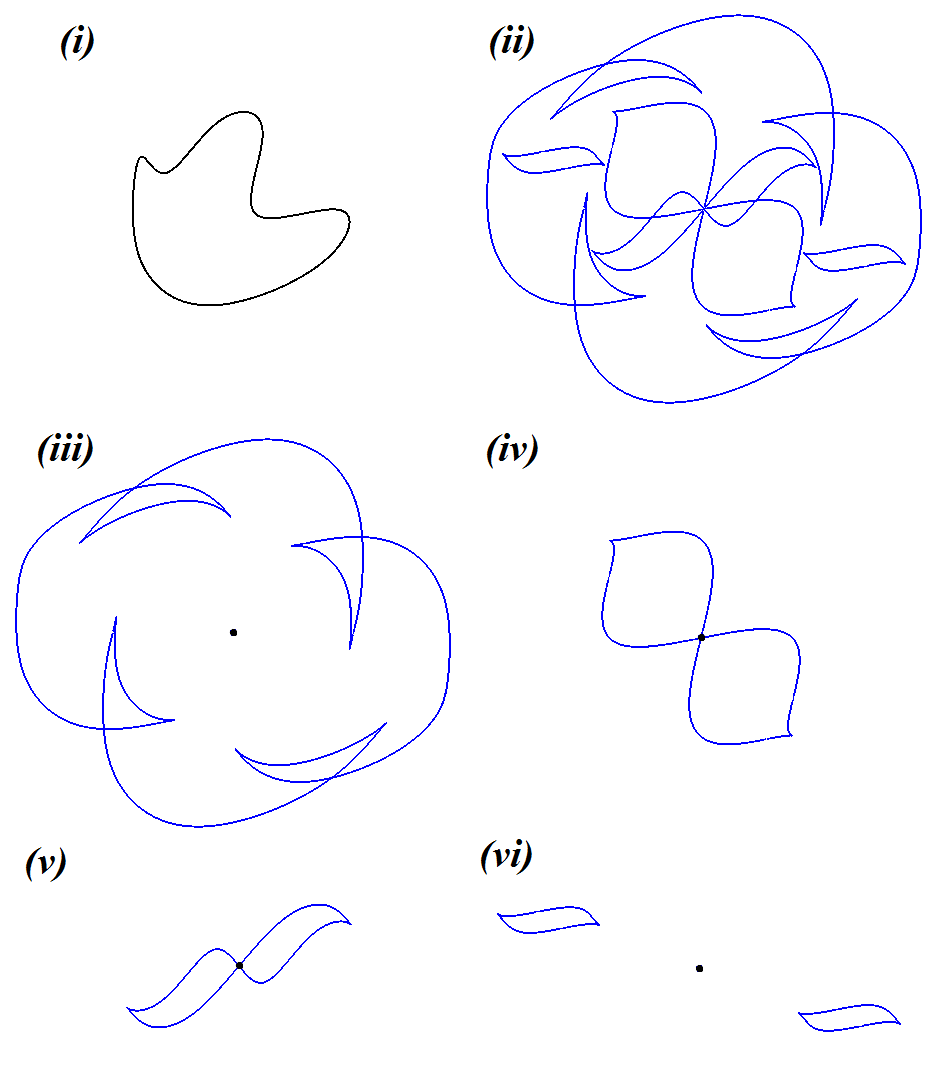}
\caption{(i) A non-convex curve $M$ with four inflexion points, (ii) $\SC(M)$, (iii) a smooth branch of $\SC(M)$, (iv)-(v) smooth branches of $\SC(M)$ passing through the origin (the marked point), (vi) two smooth branches of $\SC(M)$}
\label{PictureAlgExam}
\end{figure}

We assume along this section that $M$ is a generic regular closed curve. We shall adapt an algorithm that describes the geometry of smooth branches of affine equidistants  to the case of the secant caustic (see Section 3 in \cite{D-Z} for details).

\begin{defn}\label{DefAngleFunctin}
Let $S^1\ni s\mapsto f(s)\in\mathbb{R}^2$ be a parameterization of a smooth closed curve $\M$, such that $f(0)$ is not an inflexion point. A function $\varphi_{\M}:S^1\to [0,\pi]$ is called an \textit{angle function of} $\M$ if $\varphi_{\M}(s)$ is the oriented angle between $f'(s)$ and  $f'(0)$ modulo $\pi$. \change{In this sense the set $[0,\pi]$ is $S^1$ with identification modulo $\pi$.}
\end{defn}

\begin{defn}\label{DefLocalExtrema}
A point $\varphi$ in $S^1$ is a \textit{local extremum }of $\varphi_{\M}$ if there exists $s$ in $S^1$ such that $\varphi_{\M}(s)=\varphi$, $\varphi'_{\M}(s)=0$, $\varphi''_{\M}(s)\neq 0$. The local extremum $\varphi$ of $\varphi_{\M}$ is a \textit{local maximum (respectively minimum)} if $\varphi''_{\M}(s)<0$ (respectively $\varphi''_{\M}(s)>0$). We denote by $\mathcal{M}(\varphi_{\M})$ the set of local extrema of $\varphi_{\M}$.  
\end{defn}

\begin{figure}[h]
\centering
\includegraphics[scale=0.33]{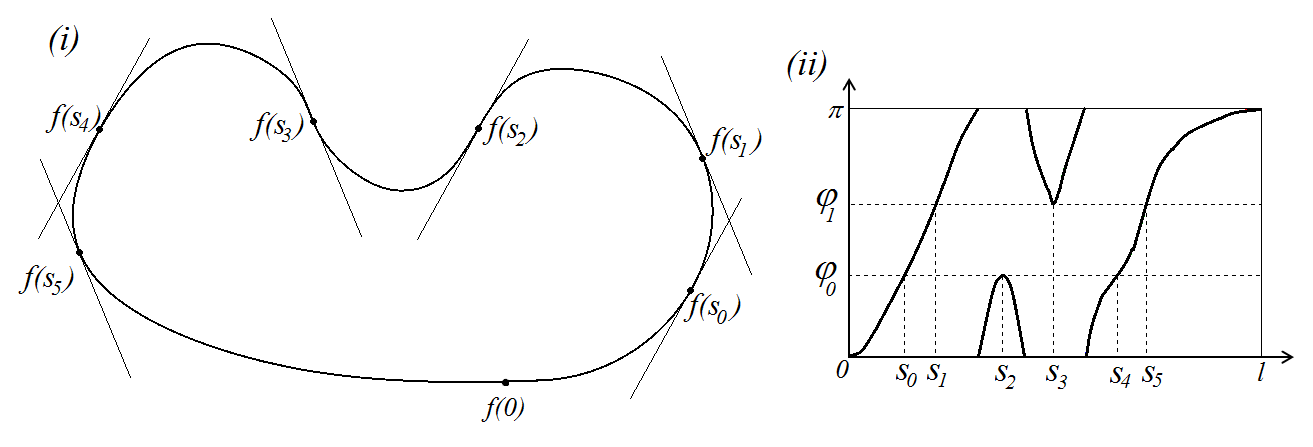}
\caption{(i) A closed regular curve $\M$ with points $f(s_i)$ and tangent lines to $\M$ at these points, (ii) a graph of the angle function $\varphi_{\M}$ with $\varphi_i$ and $s_i$ values}
\label{PictureAlgFindingPhiSets}
\end{figure}

It is easy to see that if $f$ is the arc length parameterization of a generic regular closed curve $M$ then $f(s_1)$, $f(s_2)$ is a parallel pair of $M$ if and only if $\varphi_M(s_1)=\varphi_M(s_2)$. Furthermore $M$ has an inflexion point at $f(s_0)$ if and only if $\varphi_M(s_0)$ is a local extremum. The function $\varphi_M$ has an even number of local extrema. After the local maximum the next extremum is a local minimum, and vice versa. 

\begin{defn}\label{DefSeqOfLocExtr}
The set of local extrema $\mathcal{M}(\varphi_{\M})= \{\varphi_0, \varphi_1, \ldots, \varphi_{2n-1}\}$ with the order compatible with the orientation of $S^1=\varphi_{\M}(S^1)$ will be called a \textit{sequence of local extrema}. 
\end{defn}

\begin{defn}\label{DefSeqOfParallPts}
The \textit{sequence $\mathcal{S}_{\M}$ of parallel points} is defined as the subset \linebreak $ \varphi_{\M}^{-1}\left(\mathcal{M}(\varphi_{\M})\right)= \{s_0, s_1, \ldots, s_{m-1}\}$, 
 provided $\mathcal{M}(\varphi_{\M})$ is not empty; otherwise we shall take $\{s_0, s_1, \ldots, s_{m-1}\}=\varphi_{\M}^{-1}\left(\varphi_{\M}(0)\right)$. The  points of $\mathcal{S}_{\M}$ are ordered according to the orientation of $\M$.   
\end{defn}

In Fig. \ref{PictureAlgFindingPhiSets} we illustrate an example of a closed regular curve $\M$ with corresponding angle function $\varphi_{\M}$.

\medskip
We observe that the number of elements in the sequence $\mathcal{S}_M$ is even (see Proposition 3.7 in \cite{D-Z}). For the remaining part of this section we set $2m=\#\mathcal{S}_{\M}$.

We define
\begin{align*}
\minm_{2m}(k,l):&=\left\{\begin{array}{ll}2m-1, &\text{ if }\{k,l\}=\{0,2m-1\},\\ \min(k,l), &\text{ otherwise},\end{array}\right.\\
\maxm_{2m}(k,l):&=\left\{\begin{array}{ll}0, &\text{ if }\{k,l\}=\{0,2m-1\},\\ \max(k,l), &\text{ otherwise}.\end{array}\right.
\end{align*}

An interval $(s_{2m-1},s_0)$ denotes the interval $(s_{2m-1}, L_{\M}+s_0)$, where $L_{\M}$ is the length of $M$.

\medskip
In the following definition the indexes $i$ in $\varphi_i$ are computed modulo $2n$ and the numbers $j, j+1$ in the pairs $(j,j+1)$ and $(j+1,j)$ are computed modulo $2m$.

\begin{defn}\label{DefSetParallArcs}
If $\mathcal{M}(\varphi_{\M})=\{\varphi_0, \varphi_1, \ldots, \varphi_{2n-1}\}$, then for every $i\in\{0, 1, \ldots, 2n-1\}$, a \textit{set of parallel arcs} $\Phi_i$ is the following subset
\begin{align*}
\Phi_i=\Big\{\overarc{p_k}{p_l}\ \big|\ &k-l=\pm 1\ \mbox{mod} (m),\ \varphi_{\M}(s_k)=\varphi_i,\ \varphi_{\M}(s_l)=\varphi_{i+1}, \\ &\varphi_{\M}\big((s_{\minm_{2m}(k,l)}, s_{\maxm_{2m}(k,l)})\big)=(\varphi_i, \varphi_{i+1})\Big\},
\end{align*}
where $p_i:=f(s_i)$ and $\overarc{p_k}{p_l}:=f\left(\big[s_{\minm_{2m}(k,l)}, s_{\maxm_{2m}(k,l)}\big]\right)$.

\medskip
If $\mathcal{M}(\varphi_{\M})$ is empty then we define only one \textit{set of parallel arcs} as follows: 
\begin{align*}
\Phi_0=\big\{\overarc{p_0}{p_1}, \overarc{p_1}{p_2}, \ldots, \overarc{p_{2m-2}}{p_{2m-1}}, \overarc{p_{2m-1}}{p_0}\big\}.
\end{align*}
\end{defn}

The set of parallel arcs has the following property.

\begin{prop}\label{PropDiffeoBetweenParallelArcs}
Let $f:S^1\to\mathbb{R}^2$ be the arc length parameterization of $\M$.
For every two arcs $\overarc{p_k}{p_l}$, $\overarc{p_{k'}}{p_{l'}}$ in the same set of parallel arcs, the well defined map
\begin{align*}
\overarc{p_k}{p_l}\ni p\mapsto P(p)\in \overarc{p_{k'}}{p_{l'}},
\end{align*}
is a diffeomorphism, where the pair $p, P(p)$ is a parallel pair of $\M$.
\end{prop}

As a consequence, we obtain the following result.

\begin{thm}\label{EqAsSumThm}
If $f:S^1\to\mathbb{R}^2$ is the arc length parameterization of $M$ then 
\begin{align}\label{EqUnionOfArcs}
\SC( M)=\bigcup_{i}\bigcup_{\substack{\overarc{p_k}{p_l}, \overarc{p_{k'}}{p_{l'}}\in\Phi_i \\ \overarc{p_k}{p_l}\neq\overarc{p_{k'}}{p_{l'}}}}\SC\left(\overarc{p_k}{p_l}\cup\overarc{p_{k'}}{p_{l'}}\right).
\end{align}
\end{thm}

\begin{defn}\label{DefGlueingScheme}
Let $\overarc{p_{k_1}}{p_{k_2}}$, $\overarc{p_{l_1}}{p_{l_2}}$ belong to the same set of parallel arcs, then $\begin{array}{ccc}
p_{k_1} &\frown& p_{k_2} \\ \hline
p_{l_1} &\frown&p_{l_2} \\ \hline \end{array}$ denotes the following arc 
\begin{align*}
\textrm{cl}\Big\{(a, b)\in M\times M\ \Big|\ &a\in \overarc{p_{k_1}}{p_{k_2}}, b\in\overarc{p_{l_1}}{p_{l_2}}, a,b\text{ is a parallel pair of }\M\Big\}.
\end{align*}

In addition $\begin{array}{ccccc}
p_{k_1} &\frown &\ldots &\frown & p_{k_n}\\ \hline
p_{l_1} &\frown &\ldots &\frown& p_{l_n}\\ \hline \end{array}$ denotes $\displaystyle\bigcup_{i=1}^{n-1}\begin{array}{ccc} 
p_{k_i} &\frown &p_{k_{i+1}} \\ \hline
p_{l_i} &\frown &p_{l_{i+1}} \\ \hline \end{array}$. We will call this union of arcs a \textit{glueing scheme} for $\SC(\M)$.
\end{defn}

\begin{defn}
The \textit{secant map} of the curve $M$ is the following map:
\begin{align*}
S_M: M\times M\to\mathbb{R}^2, (a, b)\mapsto a-b.
\end{align*}
\end{defn}

Let $\mathcal{A}_1=\overarc{\p_{k_1}}{\p_{k_2}}$ and $\mathcal{A}_2=\overarc{\p_{l_1}}{\p_{l_2}}$ be two arcs of $M$ which belong to the same set of parallel arcs. It is easy to see that the set $\SC\Big(\mathcal{A}_1\cup\mathcal{A}_2\Big)$ consists of the image of two different arcs $\begin{array}{ccc}
\p_{k_1} &\frown&\p_{k_2} \\ \hline
\p_{l_1} &\frown&\p_{l_2} \\ \hline \end{array}$ and $\begin{array}{ccc}
\p_{l_1} &\frown&\p_{l_2} \\ \hline 
\p_{k_1} &\frown&\p_{k_2} \\ \hline
\end{array}$ under the secant map $S_M$. From this observation we get the following proposition.

\begin{cor}\label{CorNumDiffArcs}
The set $\SC(\M)$ is the image of the union of $\displaystyle 2\cdot\sum_{i}{\#\Phi_i\choose 2}$ different arcs under the secant map $S_M$.
\end{cor}

The algorithm glues arcs of $\SC( M)$ corresponding to pairs of parallel arcs of $M$ in order to create branches of $\SC( M)$. 

\begin{prop}[see Proposition 3.15 in \cite{D-Z}]\label{PropAlgAlwaysGoFurhter}
Let $M$ be a generic regular closed curve which is not convex. If a glueing scheme for $\SC( M)$ is of the form \linebreak $\begin{array}{ccc} 
p_{k_1} &\frown &p_{k_2} \\ \hline
p_{l_1} &\frown &p_{l_2} \\ \hline \end{array}$, then this scheme can be prolonged in a unique way to \begin{align*}\begin{array}{ccccc} 
p_{k_1} &\frown &p_{k_2} &\frown &p_{k_3}\\ \hline
p_{l_1} &\frown &p_{l_2} &\frown &p_{l_3}\\ \hline \end{array}\end{align*}
such that the pair $(k_1,l_1)$ is different than the pair $(k_3,l_3)$.
\end{prop}

The glueing scheme represents parts of branches of $\SC( M)$. If we equip the set of all possible glueing schemes for $\SC( M)$ with the inclusion relation, then this set is partially ordered. The maximal glueing schemes for the secant caustic are the same as the maximal glueing schemes for affine $\lambda$-equidistants for $\lambda\neq 0, \frac{1}{2}, 1$. Thus we define them in the following way.

\begin{defn}
A \textit{maximal glueing scheme} for $\SC( M)$ is a glueing scheme which is a maximal element of the set of all glueing schemes for $\SC( M)$ equipped with the inclusion relation.
\end{defn}

\begin{rem}\label{RemGlueSchemeIsBranch}
Every maximal glueing scheme corresponds to a branch of $\SC( M)$.
\end{rem}

\begin{lem}[see Lemma 3.20 in \cite{D-Z}]\label{LemPropMaxGlueSchemes}
Let $f:S^1\mapsto\mathbb{R}^2$ be the arc length parameterization of $M$. Then
\begin{enumerate}[(i)]
\item for every two different arcs $\overarc{p_{k_1}}{p_{k_2}}$, $\overarc{p_{l_1}}{p_{l_2}}$ in $\Phi_i$ there exists exactly one maximal glueing scheme for $\SC( M)$ containing $\begin{array}{ccc}
p_{k_1} &\frown &p_{k_2} \\ \hline
p_{l_1} &\frown &p_{l_2} \\ \hline \end{array}$ or  \linebreak $\begin{array}{ccc} 
p_{k_2} &\frown &p_{k_1} \\ \hline
p_{l_2} &\frown &p_{l_1} \\ \hline \end{array}$,
\medskip
\item if $p_k:=f(s_k)$ is an inflexion point of $\M$, then there exists a maximal glueing scheme for $\SC( M)$ which is in the form 
\begin{align*}
\begin{array}{ccccccccccccccccc}
p_k & \frown & p_{k_1} & \frown & \ldots & \frown & p_{k_n} & \frown & p_l & \frown & p_{l_n} & \frown & \ldots & \frown & p_{l_1} & \frown & p_k \\ \hline
p_k & \frown & p_{l_1} & \frown & \ldots & \frown & p_{l_n} & \frown & p_l & \frown & p_{k_n} & \frown & \ldots & \frown & p_{k_1} & \frown & p_k \\ \hline
\end{array},
\end{align*}
where $p_l$ is an inflexion point of $\M$ and $p_{k_i}\neq p_{l_i}$ for $i=1, 2, \ldots, n$.
\end{enumerate}
\end{lem}

\medskip
From Proposition \ref{PropPassOrigin}, Remark \ref{RemGlueSchemeIsBranch} and Lemma \ref{LemPropMaxGlueSchemes} we get the following theorem.

\begin{thm}\label{ThmAlgPartBetweenIflPt}
If $\M$ has $2n$ inflexion points then every branch of $\SC( M)$ is a closed curve and there exist exactly $n$ branches such that
\begin{enumerate}[(i)]
\item the origin is the center of symmetry of each such branch,
\item each branch passes through the origin twice and each time the origin is an inflexion point of this branch.
\end{enumerate}
\end{thm}

We illustrate Theorem \ref{ThmAlgPartBetweenIflPt} in Fig. \ref{PictureAlgExam}. Observe that despite of the fact that the secant caustic of $M$ has the center of symmetry, there can exist branches of $\SC( M)$ without this property, as illustrated in Fig. \ref{PictureAlgExam}.

\begin{thm}\label{ThmInflofSC}
Let $ M$ be a generic regular closed curve, let $2n>0$ be the number of inflexion points of $ M$. If $\#\mathcal{S}_{\M}=2m$ then
\begin{enumerate}[(i)]
\item the number of inflexion points of $\SC( M)$ is equal to $4m-2n$,
\item the number of inflexion points of every branch of $\SC( M)$ is even,
\item the number of inflexion points of every branch of $\SC(M)$ passing through the origin is $2$ modulo $4$.
\item the total number of inflexion points in all branches with the exception of branches passing through the origin is a multiple of $4$.
\end{enumerate}
\end{thm}
\begin{proof}
\begin{enumerate}[(i)]
\item Since the number of points in $ M$ which are parallel to inflexion points of $ M$ is equal to $2m-2n$, by Proposition \ref{CorInflofCS}, the number of inflexion points of $\SC( M)$ with the exception of inflexion points which are in the origin is equal to $2(2m-2n)$. Since we have $2n$ inflexion points of $\SC( M)$ in the origin, we end the proof of (i).

\item Since any branch of $\SC( M)$ is a closed curve with at most cusp singularities, the number of inflexion points of every branch of $\SC( M)$ is even (see Lemma 4.11 in \cite{D-Z3}).

\item Let $C$ be a branch of $\SC( M)$ passing through the origin. By Lemma \ref{LemPropMaxGlueSchemes} and Proposition \ref{PropAlgAlwaysGoFurhter} the maximal glueing scheme for $\C$ has the following form:
\begin{align*}
\begin{array}{ccccccccccccccccc}
p_k & \frown & p_{k_1} & \frown & \ldots & \frown & p_{k_n} & \frown & p_l & \frown & p_{l_n} & \frown & \ldots & \frown & p_{l_1} & \frown & p_k \\ \hline
p_k & \frown & p_{l_1} & \frown & \ldots & \frown & p_{l_n} & \frown & p_l & \frown & p_{k_n} & \frown & \ldots & \frown & p_{k_1} & \frown & p_k \\ \hline
\end{array},
\end{align*}
where $p_k$, $p_l$ are inflexion points of $\M$ and $p_{k_i}\neq p_{l_i}$ for $i=1, 2, \ldots, n$.

By Lemma 3.20 in \cite{D-Z} the maximal glueing scheme for the corresponding branch of the Wigner caustic has the following form
\begin{align*}\begin{array}{ccccccccc}
p_k & \frown & p_{k_1} & \frown & \ldots & \frown & p_{k_n} & \frown & p_l \\ \hline
p_k & \frown & p_{l_1} & \frown & \ldots & \frown & p_{l_n} & \frown & p_l \\ \hline
\end{array}.
\end{align*}

By Theorem 4.10 in \cite{D-Z} this branch of the Wigner caustic has an even number of inflexion points. It means that there is an even number of points corresponding to inflexion points of $ M$ among $p_{k_1}, \ldots, p_{k_n}, p_{l_1}, \ldots, p_{l_n}$. Therefore by Proposition \ref{CorInflofCS} the number of inflexion points of $\C$ including $\begin{array}{c}  p_k \\ \hline p_k \\ \hline\end{array}$ and $\begin{array}{c} p_l \\ \hline p_l \\ \hline\end{array}$ is $2$ modulo $4$.

\item It is a consequence of (i) and (iii) and Theorem \ref{ThmAlgPartBetweenIflPt}.
\end{enumerate}
\end{proof}

\begin{defn}\label{DefTangentSC}
The tangent line of $\SC(M)$ (respectively of $\Eq_{\lambda}(M)$) at a cusp point $p$ is the limit of a sequence $T_{q_n}\SC(M)$ in $\mathbb{R}P^1$ for any sequence $q_n$ of a regular points of $\SC(M)$ (respectively of $\Eq_{\lambda}(M)$) converging to $p$.
\end{defn}

\begin{defn}
Let $\mathbbm{n}: M\to S^1$ be a continuous unit normal vector field to $M$. A vector field $\mathbbm{n}_{\SC}: \SC(M)\ni a-b\mapsto \mathbbm{n}(a)\in S^1$ for every parallel pair $a,b$ of $M$ is called a \textit{normal vector field} to $\SC(M)$.
\end{defn}

It is easy to see that $\mathbbm{n}_{\SC}$ is continuous on every branch of $\SC(M)$ and the vector $\mathbbm{n}_{\SC}(a-b)$ is perpendicular to the tangent space to $\SC(M)$ at $a-b$.

\begin{defn}
A \textit{rotation number} of a curve with at most cusp singularities is a rotation number of its continuous unit normal vector field.
\end{defn}

This definition coincides with the classical definition of the rotation number for regular curves.

\begin{thm}\label{ThmAlgEvenNoOfCusp}
If $C$ is a smooth branch of $\SC( M)$ then the number of cusps of $C$ is even.
\end{thm}
\begin{proof}
Let $\mathbbm{n}_{\SC}$ be a normal vector field to $C$. The vector field $\mathbbm{n}_{\SC}$ is continuous and normal to the cusp singularity. Thus it is directed outside the cusp on the one of two connected regular components and is directed inside the cusp on the other component as it is illustrated in Fig. \ref{FigNormalVFtoCusp}. Since $C$ is a closed curve and $\mathbbm{n}_{\SC}$ is continuous, the rotation number of $C$ is an integer. Therefore the number of cusps of $C$ is even.

\begin{figure}[h]
\centering
\includegraphics[scale=0.15]{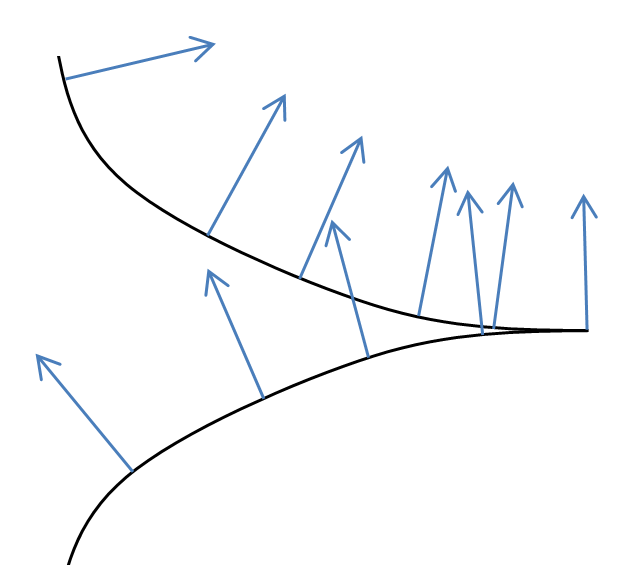}
\caption{A cusp singularity with a continuous normal vector field}
\label{FigNormalVFtoCusp}
\end{figure}
\end{proof}

\begin{prop}
Let $C$ be a branch of $\SC(M)$ passing through the origin. An oriented half-branch of $C$ between two inflexion points at the origin has even number of cusps if and only if the tangent and normal vector fields at the beginning and the end of this half-branch define the same orientation of $T_{(0,0)}\mathbb{R}^2$.
\end{prop}
\begin{proof}
It is a consequence of a fact that the tangent vector field changes the orientation after crossing the cusp point (see Fig. \ref{FigCoorCusp}).

\begin{figure}[h]
\centering
\includegraphics[scale=0.20]{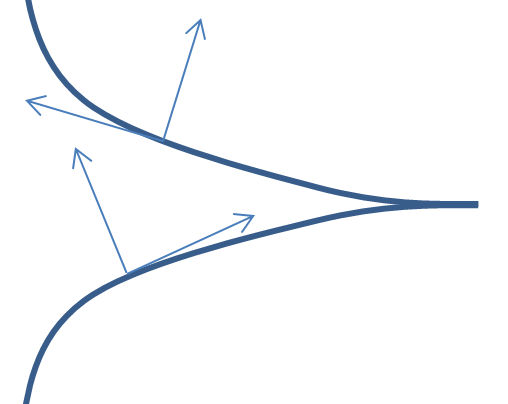}
\caption{A cusp singularity}
\label{FigCoorCusp}
\end{figure}
\end{proof}

\begin{defn}
We say that two arcs $\overarc{p_k}{p_l}$ and $\overarc{p_{k'}}{p_{l'}}$ in $\Phi_i$ are \textit{curved in the same side} (respectively \textit{curved in the different sides}) if these arcs are curved in the same side (respectively curved in the different sides) at every parallel pair $a,b$ such that $a\in\overarc{p_k}{p_l}-\{p_k, p_l\}$, $b\in\overarc{p_{k'}}{p_{l'}}-\{p_{k'}, p_{l'}\}$.
\end{defn}

By Corollary \ref{CorThmCurvPosWrinkled} we obtain the following useful observation.

\begin{cor}
Let $M$ be a generic regular closed curve. If there exist arcs curved in the same side $\overarc{p_k}{p_l}$ and $\overarc{p_{k'}}{p_{l'}}$ such that $p_k$, $p_{l'}$ or $p_{k'}$, $p_l$ are inflexion points of $M$ then $\SC(M)$ has at least two cusps.
\end{cor}

\section{The geometry of the secant caustic of rosettes}

Let $L_C$ denote the length of a curve $C$. Let $A_C$ denote the area of the region bounded by a simple curve $C$. Let $\widetilde{A}_C$ denote the \textit{signed area} (or an \textit{oriented $\slash$ algebraic area}) of a closed oriented curve $C$, i.e. the integral 
$$\displaystyle\frac{1}{2}\int_C-y\,\mathrm{d}x+x\,\mathrm{d}y=\iint_{\mathbb{R}^2}w_C(x, y)\,\mathrm{d}x\,\mathrm{d}y,$$ 
where $w_C(x,y)$ is the winding number of $C$ around a point $(x, y)$.

\begin{defn}
A smooth regular oriented closed curve is called an $n$-\textit{rosette} if its curvature is positive and its rotation number is equal to $n$.
\end{defn}

Let $R_n$ be an $n$-rosette and let a point $\bold{0}$ be the origin of $\mathbb{R}^2$. Let 
\begin{align*}
[0, 2n\pi]\ni\theta\mapsto\gamma(\theta)\in\mathbb{R}^2
\end{align*} be a parameterization of $R_n$ in terms of the tangential angle $\theta$ to $R_n$. \change{Please notice that $\gamma(0)=\gamma(2n\pi)$. Therefore, we can identify $[0,2n\pi]$ with $S^1$ modulo $2n\pi$.} We will use a special parameterization which is based on a notion of a generalized support function $[0,2n\pi]\ni\theta\mapsto p(\theta)\in\mathbb{R}$. Geometrically $p(\theta)$ is an oriented distance between the origin $\bold{0}$ and the tangent line to $R_n$ at a point $\gamma(\theta)$ in the direction $(\cos\theta, \sin\theta)$. Since $R_n$ is an envelope of the family of tangent lines to it, one can easily get that the parameterization of $R_n$ in terms of a couple $(\theta, p(\theta))$ is as follows:
\begin{align*}
[0, 2n\pi]\ni\theta\mapsto\gamma(\theta):=\big(p(\theta)\cos\theta-p'(\theta)\sin\theta, p(\theta)\sin\theta+p'(\theta)\cos\theta)\in\mathbb{R}^2.
\end{align*}

The couple $(\theta, p(\theta))$ is called the polar-tangential coordinates of an $n$-rosette and it is very useful to study the convex and locally convex objects. For details see \cite{CM}.

The curvature and the radius of curvature of $R_n$ at a point $\gamma(\theta)$ are given by the formulas
\begin{align*}
\kappa(\theta)=\frac{1}{p(\theta)+p''(\theta)},\quad \rho(\theta)=p(\theta)+p''(\theta).
\end{align*}

The length and the oriented area of $R_n$ can be computed as follows:
\begin{align}
\label{CauchyFormula} L_{R_n}&=\int_0^{2n\pi}\rho(\theta)\,\mathrm{d}\theta=\int_0^{2n\pi}p(\theta)\,\mathrm{d}\theta,\\
\label{BlaschkeFormula} \widetilde{A}_{R_n}&=\frac{1}{2}\int_0^{2n\pi}\Big(p^2(\theta)-p'^2(\theta)\Big)\,\mathrm{d}\theta.
\end{align}
Formulas \eqref{CauchyFormula}, \eqref{BlaschkeFormula} are known as Cauchy's and Blaschke's formulas, respectively \cite{Gro}.

Geometrical objects related with rosettes were studied in many papers (see \cite{CM, D-Z, MM1, Z3, Z1} and the literature therein). Rosettes are also planar non-singular hedgehogs, i.e. curves which can be parameterized using their Gauss map. Hedgehogs can be viewed as Minkowski's difference of convex bodies (see \cite{MMa1}). In \cite{D-Z, Z1} the geometry of affine $\lambda$-equidistants of rosettes were studied it was shown that the Wigner caustic of a generic  $n$-rosette $R_n$ has exactly $n$ smooth branches:
\begin{itemize}
\item a branch $\Eq_{\frac{1}{2}, k}(R_n)$ for $k=1, \ldots, n-1$, which has the following parameterization:
\begin{align*}
[0,2n\pi]\ni\theta\mapsto\gamma_{\frac{1}{2}, k}(\theta):=\frac{1}{2}\left(\gamma(\theta)+\gamma(\theta+k\pi)\right)\in\mathbb{R}^2,
\end{align*}
\item a branch $\Eq_{\frac{1}{2}, n}(R_n)$, which has the following parameterization:
\begin{align*}
[0,n\pi]\ni\theta\mapsto\gamma_{\frac{1}{2}, n}(\theta):=\frac{1}{2}\left(\gamma(\theta)+\gamma(\theta+n\pi)\right)\in\mathbb{R}^2,
\end{align*}
\end{itemize}
where $[0,2n\pi]\ni\theta\mapsto\gamma(\theta)\in\mathbb{R}^2$ is a parameterization of a rosette $R_n$ in the polar-tangential coordinates $\big(p(\theta), \theta\big)$.

Let $\mathcal{R}_{\bold{0}}$ denote a point reflection through the origin of $\mathbb{R}^2$.

\begin{thm}\label{ThmRosettes}
Let $R_n$ be a generic $n$-rosette and let $[0,2n\pi]\ni\theta\mapsto\gamma(\theta)\in\mathbb{R}^2$ be a parameterization of $R_n$ in the polar-tangential coordinates $\big(p(\theta), \theta\big)$. Then
\begin{enumerate}[(i)]
\item there are $2n-1$ branches of the secant caustic of $R_n$:
\begin{itemize}
\item a branch $\SC_k(R_n)$ which has a parameterization 
\begin{align}
\label{SCparamRosette1}[0, 2n\pi]\ni\theta\mapsto\gamma_{k, n}(\theta):=\gamma(\theta)-\gamma(\theta+k\pi)\in\mathbb{R}^2
\end{align} 
for $k=1, 2, \ldots, n$,
\item a branch $\SC_{n+k}(R_n)$ which has a parameterization 
\begin{align}
\label{SCparamRosette2}[0, 2n\pi]\ni\theta\mapsto\gamma_{n+k, n}(\theta):=\gamma\big(\theta+k\pi\big)-\gamma(\theta)\in\mathbb{R}^2
\end{align}
for $k=1, 2, \ldots n-1$,
\end{itemize}
\item for each $k=1, 2, \ldots, n-1$ we have $\mathcal{R}_{\bold{0}}\big(\SC_k(R_n)\big)=\SC_{n+k}(R_n)$,
\item the branch $\SC_n(R_n)$ is centrally symmetric,
\item the rotation number of each branch of $\SC(R_n)$ is equal to $n$,
\item exactly $n$ branches of $\SC(R_n)$ are rosettes:
\begin{itemize}
\item $\SC_{k}(R_n)$ for $k$ odd if $n$ is even,
\item $\SC_{k}(R_n)$, $\SC_{n+k}(R_n)$ for $k$ odd and smaller than $n$ and $\SC_{n}(R_n)$ if $n$ is odd,
\end{itemize}
\item exactly $n-1$ branches of $\SC(R_n)$ are singular:
\begin{itemize}
\item $\SC_k(R_n)$ for $k$ even if $n$ is even,
\item $\SC_{k}(R_n)$, $\SC_{n+k}(R_n)$ for $k$ even and smaller than $n$ if $n$ is odd,
\end{itemize}
and the number of cusps in each singular branch is even,
\item the minimal number of cusps of $\SC(R_n)$ is $2(n-1)$,
\item if $\mathcal{C}$ is a non-singular branch of $\SC(R_n)$, then $L_{\mathcal{C}}=2L_{R_n}$,
\item if $\mathcal{C}$ is a singular branch of $\SC(R_n)$, then $L_{\mathcal{C}}\leqslant 2L_{R_n}$,
\item if $k<n$, then
\begin{align*}
\widetilde{A}_{\SC_k(R_n)}+4\widetilde{A}_{\Eq_{\frac{1}{2}, k}(R_n)} &=4\widetilde{A}_{R_n},\\
\widetilde{A}_{\SC_{n+k}(R_n)}+4\widetilde{A}_{\Eq_{\frac{1}{2}, k}(R_n)} &=4\widetilde{A}_{R_n},
\end{align*}
and if $k=n$, then
\begin{align*}
\widetilde{A}_{\SC_k(R_n)}+8\widetilde{A}_{\Eq_{\frac{1}{2}, k}(R_n)} &=4\widetilde{A}_{R_n}.
\end{align*}
\end{enumerate}
\end{thm}
\begin{proof}
The set of parallel arcs has the following form
$$\Phi_0=\left\{\overarc{\p_0}{\p_1}, \overarc{\p_1}{\p_2}, \ldots, \overarc{\p_{2n-2}}{\p_{2n-1}}, \overarc{\p_{2n-1}}{\p_0}\right\}.$$

Let $\SC_{k}(R_n)$ be a smooth branch of $\SC(R_n)$. We can create the following maximal glueing schemes.
\begin{itemize}
\item A maximal glueing scheme of $\SC_k(R_n)$ for $k\in\{1,2,\ldots,n-1\}$:
$$\begin{array}{ccccccccccccc} 
\p_0	&\frown&	\p_1	&\frown&	\p_2	&\frown&	\ldots 	&\frown&	\p_{2n-2}	&\frown&	\p_{2n-1} 	&\frown&	\p_0 	\\ \hline 
\p_k	&\frown&	\p_{k+1}	&\frown&	\p_{k+2}	&\frown&	\ldots 	&\frown&	\p_{k-2}	&\frown&	\p_{k-1}	&\frown&	\p_k	\\ \hline
\end{array}.$$

\item A maximal glueing scheme of $\SC_n(R_n)$:
$$\begin{array}{ccccccccccccccc}
\p_0	&\frown&	\p_1	&\frown&	\ldots 	&\frown&	  \p_{n-1} 	&\frown& \p_n	&\frown&	\ldots 	&\frown& 	\p_{2n-1} 	&\frown& \p_0	\\ \hline
\p_n	&\frown&	\p_{n+1} &\frown&	\ldots 	&\frown&	\p_{2n-1} 	&\frown& \p_0 		&\frown& 	\ldots 	&\frown& 	\p_{n-1}	&\frown& 	\p_n	\\ \hline
\end{array}.$$

\item A maximal glueing scheme of $\SC_{n+k}(R_n)$ for $k\in\{1,2,\ldots,n-1\}$:
$$\begin{array}{ccccccccccccc} 
\p_k	&\frown&	\p_{k+1}	&\frown&	\p_{k+2}	&\frown&	\ldots 	&\frown&	\p_{k-2}	&\frown&	\p_{k-1}	&\frown&	\p_k	\\ \hline
\p_0	&\frown&	\p_1	&\frown&	\p_2	&\frown&	\ldots 	&\frown&	\p_{2n-2}	&\frown&	\p_{2n-1} 	&\frown&	\p_0 	\\ \hline 
\end{array}.$$
\end{itemize}

The total number of arcs of the glueing schemes for the secant caustic of $R_n$ presented above is $2n(2n-1)$. By Corollary \ref{CorNumDiffArcs} the total number of different arcs of the secant caustic of $R_n$ is equal to the same number. Thus there is no more maximal glueing schemes for the secant caustic of $R_n$. Therefore there are exactly $2n-1$ branches of $\SC(R_n)$ which are parameterized as in $(i)$. By parameterizations (\ref{SCparamRosette1}) and (\ref{SCparamRosette2}) we get (ii), (iii) and (iv).

Let $\rho(\theta)$ denote the radius of curvature of $R_n$ at a point $\gamma(\theta)$. Let us recall that $\rho(\theta)=p(\theta)+p''(\theta)$.

Since $p(\theta)$ is a support function of $R_n$, directly by (\ref{SCparamRosette1}) and (\ref{SCparamRosette2}) we get that
\begin{itemize}
\item the support function and the radius of curvature of $\SC_k(R_n)$ for $k=1, 2, \ldots, n$ are given by the following formulas:
\begin{align*}
p_{k, n}(\theta)&=p(\theta)+(-1)^{k+1}p(\theta+k\pi),\\
\rho_{k, n}(\theta)&=\rho(\theta)+(-1)^{k+1}\rho(\theta+k\pi).
\end{align*}
\item the support function and the radius of curvature of $\SC_k(R_n)$ for $k=n+1, \ldots, 2n-1$ are given by following formulas:
\begin{align*}
p_{k, n}(\theta) &= (-1)^{k-n}p(\theta+(k-n)\pi)-p(\theta),\\
\rho_{k, n}(\theta) &=(-1)^{k-n}\rho(\theta+(k-n)\pi)-\rho(\theta). 
\end{align*}
\end{itemize}

We will prove points (v) and (vi) only when $n$ is even, the proof for the remaining case is similar. Let us notice that $\SC_k(R_n)$ is singular at a point $\gamma_{k, n}(\theta)$ if and only if $\rho_{k, n}(\theta)=0$. Moreover let us notice that if $k$ is odd and $k<n$, then $\rho_{k, n}(\theta)=\rho(\theta)+\rho(\theta+k\pi)>0$ and $\rho_{n+k, n}(\theta)=-\rho(\theta+(k-n)\pi)-\rho(\theta)<0$. Hence $\SC_k(R_n)$ is an $n$-rosette if $k$ is odd. Now let us assume that $k$ is even. Since in this case $\rho_{k, n}(0)=\rho_{k, n}(2n\pi)$ and $\displaystyle\int_0^{2n\pi}\rho_{k, n}(\theta)\,\mathrm{d}\theta=0$, there are even number of zeros of the function $\rho_{k, n}$ in an interval of the length $2n\pi$.

One can check that $\displaystyle p(\theta)=2+\cos\frac{\theta}{n}$ is a support function of an $n$-rosette such that each singular branch of the secant caustic of $R_n$ has exactly $2$ cusps. This ends the proof of (vii).

To prove (viii) let us notice that if $\SC_k(R_n)$ is a rosette, then $|\rho_{k, n}(\theta)|=\rho(\theta)+\rho(\theta+m_k\pi)$, where $m_k$ is some integer depending on $k$. Hence
\begin{align*}
L_{\SC_k(R_n)}&=\int_0^{2n\pi}|\rho_{k, n}(\theta)|\,\mathrm{d}\theta=\int_0^{2n\pi}(\rho(\theta)+\rho(\theta+m_k\pi))\,\mathrm{d}\theta=L_{R_n}+L_{R_n}.
\end{align*}
To prove (ix) let us notice that if $\SC_k(R_n)$ is a singular hedgehog, then $|\rho_{k,n}|=|\rho(\theta)-\rho(\theta+m_k\pi)|$, where $m_k$ is some integer depending on $k$. Hence
\begin{align*}
L_{\SC_k(R_n)}=\int_0^{2n\pi}|\rho_{k, n}(\theta)|\,\mathrm{d}\theta=\int_0^{2n\pi}|\rho(\theta)-\rho(\theta+m_k\pi)|\,\mathrm{d}\theta\leqslant L_{R_n}+L_{R_n}.
\end{align*}

Let $k\leqslant n$. Then by generalized Blaschke formula we get that the oriented area of $\SC_k(R_n)$ is given by the following formula:
\begin{align}
\label{AreaSCk}
\widetilde{A}_{\SC_k(R_n)}&=\frac{1}{2}\int_0^{2n\pi}\left(p^2_{k, n}(\theta)-p'^2_{k, n}(\theta+k\pi)\right)\mathrm{d}\theta\\ 
\nonumber
&=2\widetilde{A}_{R_n}+2(-1)^{k+1}\Psi_{R_n},
\end{align}
where 
\begin{align*}
\Psi_{R_n}=\frac{1}{2}\int_0^{2n\pi}\left(p(\theta)p(\theta+k\pi)-p'(\theta)p'(\theta+k\pi)\right)\mathrm{d}\theta.
\end{align*}

From the calculation from the proof of Lemma 2.11 in \cite{Z1} we get the following relations:
\begin{itemize}
\item if $k<n$, then
\begin{align}
\label{AreaWC1}
\widetilde{A}_{\Eq_{\frac{1}{2}, k}(R_n)}=\frac{1}{2}\widetilde{A}_{R_n}+\frac{(-1)^k}{2}\Psi_{R_n},
\end{align}
\item and if $k=n$, then
\begin{align}
\label{AreaWC2}
2\widetilde{A}_{\Eq_{\frac{1}{2}, k}(R_n)}=\frac{1}{2}\widetilde{A}_{R_n}+\frac{(-1)^k}{2}\Psi_{R_n}.
\end{align}
\end{itemize}

By (\ref{AreaSCk}), (\ref{AreaWC1}) and (\ref{AreaWC2}) we end the proof of (x).
\end{proof}

Fig. \ref{Fig2rosette} illustrates a $2$-rosette $R_2$ (on the left) and $\SC(R_2)$ (on the right). One branch of $\SC(R_2)$ is dashed.

\begin{figure}[h]
\centering
\includegraphics[scale=0.35]{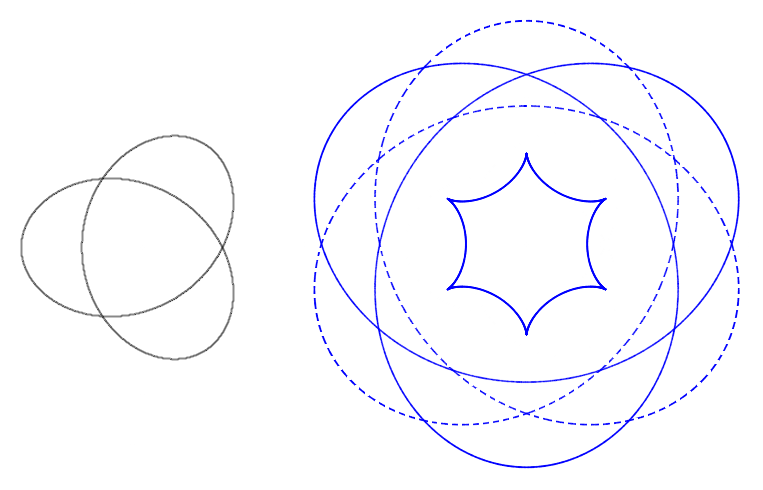}
\caption{A $2$-rosette and its secant caustic}
\label{Fig2rosette}
\end{figure}

In particular we get the following corollary of Theorem \ref{ThmRosettes} for a convex curve (let us note that if $M$ is a positively oriented oval, then $\widetilde{A}_{\Eq_{\frac{1}{2}}(M)}\leqslant 0$ \cite{Z2}).

\begin{cor}
Let $ M$ be an oval. Then $\SC( M)$ is an oval and
\begin{enumerate}[(i)]
\item
\begin{align*}
L_{\SC( M)}=2L_{ M}.
\end{align*}
\item 
\begin{align*}
A_{\SC( M)}=4A_{ M}+8|\widetilde{A}_{\Eq_{\frac{1}{2}}( M)}|.
\end{align*}
\end{enumerate}
\end{cor}

We illustrate an oval $M$, $\Eq_{\frac{1}{2}}(M)$, $\SC(M)$ in Fig. \ref{FigSCconvex}.

\begin{figure}[h]
\centering
\includegraphics[scale=0.40]{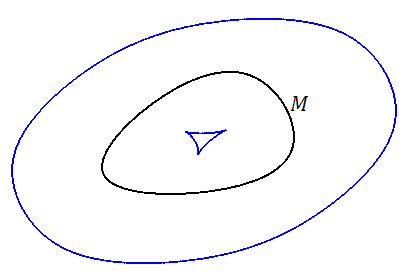}
\caption{An oval, its Wigner caustic, and its secant caustic}
\label{FigSCconvex}
\end{figure}

\bibliographystyle{amsalpha}

\end{document}